\documentclass[reqno]{amsart}
\usepackage{hyperref}
\usepackage{amssymb}

\AtBeginDocument{{\noindent\small
\emph{Electronic Journal of Differential Equations},
Vol. 2010(2010), No. 95, pp. 1--5.\newline
ISSN: 1072-6691. URL: http://ejde.math.txstate.edu or http://ejde.math.unt.edu
\newline ftp ejde.math.txstate.edu}
\thanks{\copyright 2010 Texas State University - San Marcos.}
\vspace{9mm}}

\begin{document}
\title[\hfilneg EJDE-2010/95\hfil Asymptotic properties of solutions]
{Asymptotic properties of solutions to linear nonautonomous
 delay differential equations through generalized
 characteristic equations}

\author[C. Cuevas, M. V. S. Frasson \hfil EJDE-2010/95\hfilneg]
{Claudio Cuevas, Miguel V. S. Frasson}  

\address{Claudio Cuevas \newline
Departamento de Matem\'atica,
Universidade Federal de Pernambuco,
Av.\ Prof.\ Luiz Freire, S/N,
50540-740 Recife PE, Brazil}
\email{cch@dmat.ufpe.br}

\address{Miguel V. S. Frasson \newline
Departamento de Matem\'atica Aplicada e Estat\'istica,
  ICMC -- Universidade de S\~ao~Paulo,
  Avenida Trabalhador s\~ao-carlense 400,
  13566-590 S\~ao Carlos SP, Brazil}
\email{frasson@icmc.usp.br}

\thanks{Submitted May 6, 2010. Published July 15, 2010.}
\thanks{C. Cuevas was partially supported by grant 300365/2008-0
from CNPq/Brazil}
\thanks{M. Frasson was partially supported by grant 479747/2008-3
from CNPq/Brazil}


\keywords{Functional differential equations; generalized
  characteristic equation; \hfill\break\indent
  asymptotic behavior}

\begin{abstract}
  We study some properties concerning the asymptotic behavior of
  solutions to nonautonomous retarded functional differential
  equations, depending on the knowledge of certain solutions of the
  associated generalized characteristic equation.
\end{abstract}

\maketitle
\numberwithin{equation}{section}
\newtheorem{theorem}{Theorem}[section]
\newtheorem{example}[theorem]{Example}

\section{Introduction}

We are interested in the study of the asymptotic behavior of solutions
to the linear nonautonomous retarded functional differential equation
(RFDE)
\begin{equation}
  \label{eq:fde}
  x'(t) = L(t) x_t, \quad t\geqslant t_0\in\mathbb{R},
\end{equation}
where $L(t)$ is a family of bounded linear functionals on $\mathcal{C}
= \mathcal{C}([-r,0],\mathbb{C})$, with $r>0$, depending on the
knowledge of certain solutions of the associated generalized
characteristic equation \eqref{eq:gen-chareq}, introduced below.  For
a comprehensive introduction for RFDE see \cite{HVL93}.

By the Riesz representation theorem, for each $t\geqslant t_0$, there
exists a complex valued function of bounded variation $\eta(t,\cdot)$
on $[0,r]$, normalized so that $\eta(t,0)=0$ and $\eta(t,\cdot)$ is
continuous from the right in $(0,r)$ such that
\begin{equation}
  \label{eq:L(t)=int dn(t)}
  L(t) \varphi = \int_0^r d_\theta\eta(t,\theta) \varphi(-\theta).
\end{equation}
Consider the \textit{generalized characteristic equation}
\begin{equation}
  \label{eq:gen-chareq}
  \lambda(t) = \int_0^r d_\theta \eta(t,\theta)
  \exp\Big({-\int_{t-\theta}^t} \lambda(s)ds\Big),
\end{equation}
The solutions of the generalized characteristic equation
\eqref{eq:gen-chareq} are continuous functions $\lambda(\cdot)$ defined in
$[t_0-r,\infty)$ which satisfy \eqref{eq:gen-chareq}.

One obtains the generalized characteristic equation
(\ref{eq:gen-chareq}) by looking for solutions of~\eqref{eq:fde}
the form
\begin{equation}
  \label{eq:x=exp-int-lambda}
  x(t) = \exp\Big[\int_{0}^t \lambda(s)ds\Big].
\end{equation}
For autonomous RFDE, the constant solutions of \eqref{eq:gen-chareq}
are the roots of the so called characteristic equation.

This work is motivated by Dix, Philos and Purnaras
\cite{DixPP-nonaut-rfde-05}.  These authors studied the
asymptotic behavior of solutions of nonautonomous linear function
differential equations with discrete delays
\begin{equation}
  \label{eq:fde-philos}
  x'(t) = a(t)x(t) + \sum_{j=1}^k b_j(t) x(t-\tau_j),\quad t
  \geqslant 0
\end{equation}
where the coefficients $a(\cdot)$ and $b_j(\cdot)$ are continuous real-valued
functions on $[0,\infty)$, $\tau_j>0$ for $j = 1, 2, \dots, k$ by
means of the knowledge of solutions $\lambda(t)$, defined for
$t\geqslant -r$, of the generalized characteristic equation associated
to \eqref{eq:fde-philos}
\begin{equation}
  \label{eq:gen-charec-philos}
  \lambda(t) = a(t) + \sum_{j=1}^k b_j (t)
 \exp\Big[{-\int_{t-\tau_j}^t} \lambda(s)ds \Big], \quad t\geqslant0.
\end{equation}
We also find in \cite{DixPP-nonaut-rfde-05} a description of the
development of results of the type of Theorem~\ref{thm: V<1 =>
  asymptotic}.  We would like to mention results of this type are
found in \cite{estimates-07,large} too.  Dix, Philos
and Purnaras extended their results for neutral functional
differential equations in \cite{DixPP06}.

Theorem~\ref{thm: V<1 => asymptotic} provides a generalization of
\cite[Thm.\ 2.3]{DixPP-nonaut-rfde-05}, as it can be applied for
instance for RFDE with distributed delay or discrete variable delays,
as far as the delays are unifomly bounded.  In fact,
RFDE~\eqref{eq:fde-philos} can be written in the form~\eqref{eq:fde} letting
\[
L(t)\varphi = a(t)\varphi(0) + \sum_{j=1}^k b_j(t) \varphi(\tau_j),
\quad \varphi\in \mathcal{C}.
\]
We acknowledge that Theorem~\ref{thm: V<1 => asymptotic} is obtained
by an adaptation of the proof of \cite[Thm.\
2.3]{DixPP-nonaut-rfde-05} for the more general case of
RFDE~\eqref{eq:fde}, together with ideas from \cite{estimates-07}.
We observe that \cite[Remarks 2.4, 2.5 and
2.6]{DixPP-nonaut-rfde-05} can be restated here for
RFDE~\eqref{eq:fde} without modification.

\section{Results}
\enlargethispage{\baselineskip}
\begin{theorem}\label{thm: V<1 => asymptotic}
Assume that $\lambda(t)$ is a solution of \eqref{eq:gen-chareq}
such that
\begin{equation}
  \label{eq:V<1}
  \limsup_{t\to\infty} \int_0^r \theta
 |e^{-\int_{t-\theta}^t \lambda(s)ds}|
  d_\theta|\eta|(t,\theta) < 1.
\end{equation}
Then for each solution $x$ of \eqref{eq:fde}, we have that
the limit
\begin{equation}
  \label{eq:lim x exp -int lambda = L}
   \lim_{t\to\infty} x(t) e^{-\int_{t_0}^t \lambda(s)ds}
\end{equation}
exists, and
\begin{equation}
  \label{eq:limxexp}
  \lim_{t\to\infty} \Big[ x(t) e^{-\int_{t_0}^t \lambda(s)ds}\Big]' =
  0.
\end{equation}
Furthermore,
\begin{equation}
  \label{eq:limx'exp}
  \lim_{t\to\infty} x'(t) e^{-\int_{t_0}^t \lambda(s)ds} =
  \lim_{t\to\infty} \lambda(t) x(t) e^{-\int_{t_0}^t \lambda(s)ds}
\end{equation}
if there exists the limit in the right hand side of
\eqref{eq:limx'exp}.
\end{theorem}

\begin{proof}
  Hypothesis~\eqref{eq:V<1} implies that there exists $t_1\geqslant
  t_0$ such that
  \[
  \sup_{t\geqslant t_1} \int_0^r \theta
    |e^{-\int_{t-\theta}^t \lambda(s)ds}|
    d_\theta|\eta|(t,\theta) < 1.
  \]
  Hence without loss of generality, if necessary translating the
  initial time to $t_1$, we may assume $t_0=0$ and
  \begin{equation}
    \label{eq:V<1 t0=0}
    \mu_{\lambda} :=
    \sup_{t\geqslant 0} \int_0^r \theta
    |e^{-\int_{t-\theta}^t \lambda(s)ds}|
    d_\theta|\eta|(t,\theta) < 1.
  \end{equation}
  Let $x$ be a solution of \eqref{eq:fde}, and set
  \[
  y(t) = x(t) e^{-\int_{0}^t \lambda(s)ds}, \quad t\geqslant -r.
  \]
  Differentiating $y(t)$ when $t\geqslant0$, using that $x(t)$ is a
  solution of \eqref{eq:fde}, \eqref{eq:gen-chareq} and the
  fundamental theorem of calculus, we obtain
\begin{equation}
  \begin{aligned}
    y'(t)  & = \Big(x'(t) - x(t)\lambda(t)\Big)
    e^{-\int_{0}^t \lambda(s)ds} \\
&= \Big(
      \int_0^r d_\theta\eta(t,\theta) x(t-\theta)
      -x(t) \int_0^r d_\theta \eta(t,\theta)
      e^{-\int_{t-\theta}^t \lambda(s)ds}
    \Big) e^{-\int_{0}^t \lambda(s)ds}  \\
&=
      \int_0^r d_\theta\eta(t,\theta) x(t-\theta)
      e^{-\int_{0}^{t-\theta} \lambda(s)ds}e^{-\int_{t-\theta}^t
        \lambda(s)ds} \\
&\quad
      -x(t)e^{-\int_{0}^t \lambda(s)ds} \int_0^r d_\theta \eta(t,\theta)
      e^{-\int_{t-\theta}^t \lambda(s)ds}
       \\
&=
      \int_0^r d_\theta\eta(t,\theta) y(t-\theta)
      e^{-\int_{t-\theta}^t \lambda(s)ds}
      -y(t) \int_0^r d_\theta \eta(t,\theta)
      e^{-\int_{t-\theta}^t \lambda(s)ds}
       \\
&=
    \int_0^r d_\theta\eta(t,\theta) [y(t-\theta) -y(t)]
    e^{-\int_{t-\theta}^t \lambda(s)ds}
     \\
&=
    -\int_0^r d_\theta\eta(t,\theta) \Big[\int_{t-\theta}^t y'(s) ds\Big]
    e^{-\int_{t-\theta}^t \lambda(s)ds}, \quad t\geqslant
    0. \label{eq:y'}
  \end{aligned}
\end{equation}
  As a characteristic of RFDE, we have that $y'(t)$ is continuous
  for $t\geqslant 0$, understanding the derivative at $t=0$ as the
  derivative from the right.  Let
  \begin{equation}\label{eq:def Mx}
    M_x = \max_{t\in[0,r]} |y'(t)|.
  \end{equation}
  Let $t^*\geqslant r$ arbitrary and suppose that for some $A\geqslant
  0$ we have
  \begin{equation}
    \label{eq:hipotese |y'|<A}
    |y'(t)| \leqslant A,\quad t^*-r\leqslant t\leqslant t^*.
  \end{equation}
  Using \eqref{eq:V<1 t0=0} and \eqref{eq:y'}, we estimate that
\begin{align*}
    |y'(t^*)| &
    \leqslant \Big|\int_0^r d_\theta\eta(t,\theta)
      \Big[\int_{t-\theta}^t y'(s) ds\Big]
      e^{-\int_{t-\theta}^t \lambda(s)ds}\Big|
    \\
& \leqslant \int_0^r d_\theta|\eta|(t,\theta)
      \Big|\int_{t-\theta}^t y'(s) ds\Big|
      \big|e^{-\int_{t-\theta}^t \lambda(s)ds}\big|
    \\
& \leqslant A \int_0^r \theta
    \big|e^{-\int_{t-\theta}^t \lambda(s)ds}\big| d_\theta|\eta|(t,\theta)
    \leqslant A \mu_{\lambda}.
\end{align*}
  Since $|y'(t^*)|\leqslant A\mu_{\lambda}<A$, the continuity of
  $|y'(t)|$ implies that
  \[
  |y'(t)|\leqslant A,\quad t\in[t^*-r,t^*+\delta].
  \]
Reasoning as above, we show that
  \[
  |y'(t)|\leqslant A\mu_{\lambda},\quad t\in[t^*,t^*+\delta].
  \]
  Since $t\mapsto |y'(t)|$ is uniformly continuous on compact
  intervals, we proceed in this way a finite number of steps and
  finally conclude that
  \begin{equation}
    \label{eq:|y'|< A mu}
    |y'(t)|\leqslant A\mu_{\lambda}, \quad t\in[t^*,t^*+r].
  \end{equation}
  Taking $t^*=nr$, $n$ a positive integer, considering
\eqref{eq:def Mx} for $n=1$ and using \eqref{eq:|y'|< A mu} with
  $A=M_x(\mu_{\lambda})^{n-1}$ as induction step, we have proved that
  \begin{equation}
    \label{eq:|y'|< Mx mu^n}
    |y'(t)| \leqslant M_x (\mu_{\lambda})^n,\quad t\geqslant nr.
  \end{equation}
  We observe that \eqref{eq:|y'|< Mx mu^n} allows us to conclude that
  \begin{equation}
    \label{eq:|y'|< Mx mu^(t/r-1)}
    |y'(t)| \leqslant M_x (\mu_{\lambda})^{t/r -1},\quad
    t\geqslant0.
  \end{equation}
  Letting $t\to\infty$, using \eqref{eq:|y'|< Mx mu^(t/r-1)},
 we obtain~\eqref{eq:limxexp}.

  We obtain \eqref{eq:limx'exp} by a straight
  forward application of \eqref{eq:limxexp},
  differentiating the quantity in the limit and doing simple
  computations.

  We proceed to prove~\eqref{eq:lim x exp -int lambda = L}.  The
  cases $M_x=0$ and $\mu_{\lambda}=0$ are simple, where we have
  $y(t)\to x(0)$ and $y(t)\to y(r)$ as $t\to\infty$, respectively.
  For $0<\mu_{\lambda}<1$, for $0\leqslant t\leqslant T$ we obtain
  that
  \begin{align*}
    |y(T)-y(t)| &= \Big|\int_t^T y'(s)ds\Big|\\
& \leqslant
    M_x \int_t^T (\mu_{\lambda})^{s/r-1}  ds \\
&= \frac{M_x r}{\mu_{\lambda}\ln \mu_{\lambda}}
    [(\mu_{\lambda})^{T/r}-(\mu_{\lambda})^{t/r}] \to
    0\quad\text{as } t\to\infty.
  \end{align*}
  By the Cauchy's criterion of convergence, we have that $y(t)\to
  L_x$, for some $L_x$.  This shows~\eqref{eq:lim x exp -int lambda =
    L} and completes the proof.
\end{proof}

\begin{example} \rm
  Consider the linear retarded equation with variable delay
  \begin{equation}
    \label{eq:exampl-variavel}
    x'(t) = \frac{x(t-\tau(t))}{t+c-\tau(t)},\quad t\geqslant t_0.
  \end{equation}
  where $c\in\mathbb{R}$ and $\tau:[0,\infty)\to [0,r]$ is a
  continuous function such that $t+c-\tau(t)>0$ for $ t\geqslant t_0$.
  FDE~\eqref{eq:exampl-variavel} is written in the form~\eqref{eq:fde}
  letting $\eta(t,\cdot)$ be given by $\eta(t,\theta)=0$ for
  $\theta<\tau(t)$, $\eta(t,\theta)=1/(t+c-\tau(t))$ for
  $\theta\geqslant\tau(t)$.  We have that $\theta\mapsto
  \eta(t,\theta)$ is increasing and then $|\eta| = \eta$.

  The generalized characteristic equation associated to
\eqref{eq:exampl-variavel} is given by
  \begin{equation}\label{eq:ex-var-delay-chareq}
    \lambda(t) = \frac{1}{t+c-\tau(t)}
\exp\Big[{-\int_{t-\tau(t)}^t}
      \lambda(s)ds\Big]
  \end{equation}
  and we have that a solution of \eqref{eq:ex-var-delay-chareq} is
  given by
  \begin{equation}\label{eq:ex-var-delay-lambda}
    \lambda(t) = \frac{1}{t+c}.
  \end{equation}
  For \eqref{eq:exampl-variavel} and $\lambda(t)$
in~\eqref{eq:ex-var-delay-lambda}, the left hand side of~\eqref{eq:V<1}
reads as
  \[
  \limsup_{t\to\infty} \int_0^r \theta
  |e^{-\int_{t-\theta}^t \lambda(s)ds}|
  d_\theta|\eta|(t,\theta) =
  \limsup_{t\to\infty} \frac{\tau(t)}{t+c}  =0.
  \]
  and hence the hypothesis~\eqref{eq:V<1} of
  Theorem~\ref{thm: V<1 => asymptotic} is fulfilled and herefore,
  for all solutions
  $x(t)$ of \eqref{eq:exampl-variavel}, we have that
  \begin{equation}\label{eq:ex var delay - resultados teo}
   \lim_{t\to\infty} \frac{x(t)}{t+c}\text{ exists, and }
  \lim_{t\to\infty} \Big[\frac{x(t)}{t+c}\Big]' =0.
  \end{equation}
  Manipulating further the limits in \eqref{eq:ex var delay -
    resultados teo}, we are able to state that
  $x(t) = O(t)$ and $x'(t) =o(t)$ as $t\to\infty$.
\end{example}

\begin{example} \rm
  Consider the linear FDE with distributed delay
  \begin{equation}
    \label{eq:exemplo delay distrib}
    x'(t) = \int_0^1 \frac{x(t-\theta)}{t-\theta},\quad t>1.
  \end{equation}
  We write \eqref{eq:exemplo delay distrib} in the form
\eqref{eq:fde} by setting $\eta(t,\theta) = \ln t - \ln(t-\theta)$ for
  $t>1$ and $\theta\in[0,1]$.  Since $\theta\mapsto \eta(t,\theta)$
  is an increasing function, $|\eta| = \eta$.

  The generalized characteristic equation associated
  to \eqref{eq:exemplo delay distrib} is given by
  \begin{equation}\label{eq:ex var distrib - chareq}
    \lambda(t) = \int_0^1 \frac{1}{t-\theta}
\exp\Big[{-\int_{t-\theta}^t}
      \lambda(s)ds\Big] d\theta
  \end{equation}
  which has a solution given by
  \begin{equation}
    \label{eq:lambda ex distrib}
    \lambda(t) = 1/t.
  \end{equation}
  For this $\lambda(t)$ and for $t>1$, the integral in \eqref{eq:V<1}
 reads as
  \[
  \int_0^1 \frac{\theta}{t-\theta} \exp\Big[{-\int_{t-\theta}^t}
      \frac{ds}{s}\Big] d\theta =
    \int_0^1 \frac{\theta}{t}\, d\theta = \frac{1}{2t} \to 0\quad
    \text{as } t\to\infty.
  \]
  Hence the hypothesis~\eqref{eq:V<1} of Theorem~\ref{thm: V<1 =>
    asymptotic} is fulfilled.  Again we obtain that
  \begin{equation}
    \label{eq:ex var distrib - resultados teo}
    \lim_{t\to\infty} \frac{x(t)}{t}\text{ exists},\quad
    \lim_{t\to\infty} \Big[\frac{x(t)}{t}\Big]' =0\quad
    \text{and}\quad
    \lim_{t\to\infty} \frac{x'(t)}{t} =0.
  \end{equation}
\end{example}




\end{document}